\documentclass[12pt]{amsart}

\usepackage{fullpage}
\usepackage{amsmath}
\usepackage{amsfonts}
\usepackage{amssymb}
\usepackage{amsthm}
\usepackage[all,knot,poly]{xy}
\usepackage{xypic}
\usepackage{hyperref}
\usepackage{graphicx}
\usepackage{MnSymbol}
\hypersetup{
    colorlinks=true,
    linkcolor=blue,
    filecolor=magenta,      
    urlcolor=cyan,
}
 
\input xy
\xyoption{all}

\theoremstyle{plain}

\newtheorem{theorem}{Theorem}[section]

\newtheorem{proposition}[theorem]{Proposition}

\newtheorem{definition}[theorem]{Definition}
\newtheorem{example}[theorem]{Example}

\usepackage{color}

\newcommand{\norm}[1]{\lVert #1 \rVert}

\title{A Topological Data Analysis Framework for Quantifying Necrosis in Glioblastomas}
\author{Francisco Téllez}
\email{fj.tellez@uniandes.edu.co}
\address{Universidad de Los Andes, Bogot\'{a}, Colombia.}

\author{Enrique Torres-Giese}
\email{enrique.torresgiese@twu.ca}
\address{Trinity Western University, Langley, Canada.}

\begin{document}

\begin{abstract} 
In this paper, we introduce a shape descriptor that we call "interior function". This is a Topological Data Analysis (TDA) based descriptor that refines previous descriptors for image analysis. Using this concept, we define subcomplex lacunarity, a new index that quantifies geometric characteristics of necrosis in tumors such as conglomeration. Building on this framework, we propose a set of indices to analyze necrotic morphology and construct a diagram that captures the distinct structural and geometric properties of necrotic regions in tumors. We present an application of this framework in the study of MRIs of Glioblastomas (GB). Using cluster analysis, we identify four distinct subtypes of Glioblastomas that reflect geometric properties of necrotic regions.  
\end{abstract}

\maketitle

%%%%%%%%%%%%%%%%%%%%%%%%%%%%%%
%%%%%%%%%%%%%%%%%%%%%%%%%%%%%%
%%%%%%%%%%%%%%%%%%%%%%%%%%%%%%
%%%%%%%%%%%%%%%%%%%%%%%%%%%%%%

\section{Introduction}
\label{sec:introduction}

Recent applications of Topological Data Analysis (TDA) have focused on defining shape descriptors that quantify geometric characteristics, particularly in the context of image analysis. Tools such as the Smooth Euler Characteristic ~\cite{SECTGBM} and the Persistence Homology Transform \cite{PHT} use function filtration techniques in persistent homology to generate topological summaries of shapes, demonstrating the potential of TDA as a powerful analytical tool. In this paper, we propose a TDA-based framework to analyze how a 2-dimensional shape fills its surrounding space. We apply this approach to Glioblastoma MRI scans; more precisely, we study Gliobastomas by quantifying and describing their necrotic regions.

It is well established that a high degree of necrosis in patients with Glioblastomas is correlated with poor prognosis \cite{Barker}, \cite{Raza}, \cite{Yee}. A common measure of necrosis is the necrotic fraction (see for instance, \cite{fraction}), but this does not take into account the shape of the tumor and its necrosis. On the other hand, studies suggest that the fractal dimension and lacunarity of the necrotic region's border are also predictive of prognosis in glioblastoma cases \cite{Curtin}. Our goal is to develop a framework to explore necrotic geometry that emphasizes the spatial properties such as conglomeration of necrotic regions. This approach emphasizes not only the geometry but also the topology of a tumor, rather than relying solely on area-based metrics (such as necrotic fraction) or fractal properties (such as the fractal dimension and lacunarity).

For the latter purpose, we introduce a new topology-based summary for images, which we call the "interior function." This is a numerical descriptor of images that assigns higher values to regions with greater conglomeration. Using this function, we define indices to quantify pixel clustering and propose the introduction of a diagram to characterize how a 2-dimensional image fills its surrounding space. We present an application of this framework to cluster Glioblastomas based on necrotic cell aggregation, identifying 4 distinct types of tumors.

%%
%%%%%%%%%%%%%%%%%%%%%%%%%%%%%%
%%%%%%%%%%%%%%%%%%%%%%%%%%%%%%
%%%%%%%%%%%%%%%%%%%%%%%%%%%%%%
%%%%%%%%%%%%%%%%%%%%%%%%%%%%%%

\section{The interior function}
\label{int_function}

The Persistent Homology Transform (PHT) \cite{PHT} is a tool that incorporates the use of the dot product of $\mathbb{R}^{n}$ as a filtration function for persistent homology. The PHT works with three ingredients: the unit sphere $N=\mathbb{S}^{n-1}$ as the "lens" space; the functions $f_v(x)=x\cdot v$ defined on the lens space, where $v$ is a given point in $N$; and a space $M\subset \mathbb{R}^{n}$ which is equivalent to a finite simplicial complex. Following the construction of the PHT and the proof of its stability, we define the "Interior Surface" transform by changing the lens space and the functions defined on the lens space. 

Let us denote by $\mathcal{D}$ the space of persistent diagrams, which is a metric space when equipped with the bottleneck distance $d_B$ \cite{Carlsson}. When applying filtration by sub-level sets with respect to a function $f$ we obtain a diagram for each homological degree. We will denote by $\mathcal{D}_k$ the subspace of degree  $k$ diagrams, and we will use $\mathcal{D}_k(f)$ whenever we need to make clear the dependence upon the filtration function. In this paper we will focus on spaces that are homotopy equivalent to finite simplicial complexes. This category includes finite CW-complexes and cubical complexes (see for instance \cite{Hatcher}).

\begin{definition}
    Let $M\subseteq \mathbb{R}^n$ be a finite simplicial complex. Let $d:\mathbb{R}^n\times \mathbb{R}^n\to \mathbb{R}$ be a metric. For $v\in \mathbb{R}^n$ consider the function:
    \begin{equation*}
        h_v:M\to \mathbb{R},h_v(x)=-d(v,x).
    \end{equation*}
    We define the Interior Function Transform as:
    \begin{equation*}
        IFT:\mathbb{R}^n\to \mathcal{D}, v\mapsto \mathcal{D}_{n-1}(h_v).
    \end{equation*}
\end{definition}

\begin{proposition}
    The interior function transform $IFT:\mathbb{R}^n\to \mathcal{D}$ associated to a finite simplicial complex $M\subseteq \mathbb{R}^n$ is continuous.
\end{proposition}

\begin{proof} 
    Note that for any $v\in \mathbb{R}^n$, the function $h_v$ is tame since $M$ is a finite simplicial complex. Let $v_1,v_2\in \mathbb{R}^n$. Given $x\in M$, by the reverse triangle inequality for metric spaces,
    \begin{equation*}
        |h_{v_1}(x)-h_{v_2}(x)|=|d(v_1,x)-d(v_2,x)|\leq d(v_1,v_2).
    \end{equation*}
    Therefore,
    \begin{equation*}
        \norm{h_{v_1}-h_{v_2}}_{\infty}\leq d(v_1,v_2).
    \end{equation*}
    By the Bottleneck stability theorem ~\cite{Carlsson}, 
    \begin{equation*}
        d_B(\mathcal{D}_{n-1}(h_{v_1}),\mathcal{D}_{n-1}(h_{v_2}))\leq \norm{h_{v_1}-h_{v_2}}_{\infty}.
    \end{equation*}
    Hence,
    \begin{equation*}
        d_B(IFT(v_1),IFT(v_2))=d_B(\mathcal{D}_{n-1}(h_{b_1}),\mathcal{D}_{n-1}(h_{v_2}))\leq d(v_1,v_2).
    \end{equation*}
    This proves that the interior function transform is Lipschitz continuous, therefore we can conclude that $IFT$ is a continuous function.
\end{proof}

Persistence landscapes \cite{BubenikSTDA} were introduced as a way of transforming the information given by persistence diagrams into an object better suited for statistical purposes. We will use persistence landscapes to define the interior function.  

\begin{definition}[\cite{BubenikSTDA}, p. 83]
    Let $\{(b_i,d_i)\}_{i=1}^{n}$ be a persistence diagram. The persistence Landscape of the diagram is the function $\lambda:\mathbb{N}\times \mathbb{R}\to \overline{\mathbb{R}}$ such that
    \begin{equation*}
        \lambda_k(t)=\text{$k$th largest value of $min(t-b_i,d_i-t)_+$},
    \end{equation*}
    where $c_+$ denotes $max(c,0)$.
\end{definition}

The space of persistence landscapes associated to a persistence diagram is a normed vector space with the $p$-norm for landscapes.

\begin{definition}[\cite{BubenikSTDA}, p.84]
    Let $\lambda:\mathbb{N}\times \mathbb{R}\to \overline{\mathbb{R}}$ be a persistence landscape associated to a persistence diagram $D=\{(b_i,d_i)\}_{i=1}^{n}$. The $p$-norm of a Landscape is
    \begin{equation*}
        \norm{\lambda}_p=\sum_{k=0}^{\infty}\norm{\lambda_k}_p,
    \end{equation*}
    when $1\leq p<\infty$ and
    \begin{equation*}
        \norm{\lambda}_{\infty}=\sup_{k\in \mathbb{N}}\norm{\lambda_k}_{\infty}.
    \end{equation*}
    The $p$-landscape distance of diagrams $D$ and $D^{\prime}$ is
    \begin{equation*}
        \Lambda(D,D^{\prime})=\norm{\lambda-\lambda^{\prime}}_{p}.
    \end{equation*}
\end{definition}

\begin{definition}
    Let $M\subseteq \mathbb{R}^n$ be a finite simplicial complex, $d:\mathbb{R}^n\times \mathbb{R}^n\to \mathbb{R}$ a metric and $h_v:M\to \mathbb{R}$ the function of definition 2.1. Let $\lambda^{v}:\mathbb{N}\times \mathbb{R}\to \mathbb{R}$ be the Landscape associated to the diagram $\mathcal{D}(h_v)$. We define the interior function as
    \begin{equation*}
        I:\mathbb{R}^n\to \mathbb{R}, I(x)=2\norm{\lambda^{x}}_{\infty}.
    \end{equation*}
\end{definition}

As a consequence of Proposition 5 in \cite{BubenikSTDA}, $I(x)$ is equal to the length of the longest interval in the diagram corresponding to function filtration $h_x:M\to \mathbb{R}$.

\begin{proposition}
    Let $M\subseteq \mathbb{R}^n$, $d:\mathbb{R}^n\times \mathbb{R}^n\to \mathbb{R}$ and $I:\mathbb{R}^n\to \mathbb{R}$ as in Definition 2.4. The interior surface is continuous.
\end{proposition}

\begin{proof}
    We will prove Lipschitz continuity. Let $v_1,v_2\in \mathbb{R}^n$. We have the diagrams $\mathcal{D}(h_{v_1})$, $\mathcal{D}(h_{v_2})$ and the associated persistence landscapes $\lambda^{v_1}$, $\lambda^{v_2}$. By Theorem 13 in \cite{BubenikSTDA}, we know that:
    \begin{equation*}
        \norm{\lambda^{v_1}-\lambda^{v_2}}_{\infty}= \Lambda_{\infty}\left(\mathcal{D}(h_{v_1}),\mathcal{D}(h_{v_2})\right)\leq  d_B(\mathcal{D}(h_{v_1}),\mathcal{D}(h_{v_2})).
    \end{equation*}
    By the Bottleneck Stability Theorem \cite{Carlsson},
    \begin{equation*}
        d_B(\mathcal{D}(h_{v_1}),\mathcal 
        D(h_{v_2}))\leq \norm{h_{v_1}-h_{v_2}}_{\infty}. 
    \end{equation*}
    As in Proposition 2.1, note that $\norm{h_{v_1}-h_{v_2}}_{\infty}\leq d(v_1,v_2)$. Therefore,
    \begin{equation*}
        \norm{\lambda^{v_1}-\lambda^{v_2}}_{\infty}\leq d(v_1,v_2).
    \end{equation*}
    Also, by the reverse triangle inequality for normed vector spaces,
    \begin{equation*}
        |IS(v_1)-IS(v_2)|=2|\norm{\lambda^{v_1}}_{\infty}-\norm{\lambda^{v_2}}_{\infty}|\leq 2\norm{\lambda^{v_1}-\lambda^{v_2}}_{\infty}\leq 2\norm{v_1-v_2}_{\infty}.
    \end{equation*}
    Since the interior function is Lipschitz continuous, we can conclude that the interior function is continuous.
\end{proof}

\begin{example}
    The interior surface can be computed in the case of black and white images by considering the cubical complex associated to such an image and embedding it into $\mathbb{R}^2$, for our purposes we will always consider embeddings into $[-1,1]^2$. Let us consider two images that are related to an MRI of a GBM tumour. Figure \ref{fig:fig1} displays the core and enhanced image associated to a glioblastoma MRI, with a plot of the associated interior functions.

    \begin{figure}[h]
        \centering
        \includegraphics[width=0.5\linewidth]{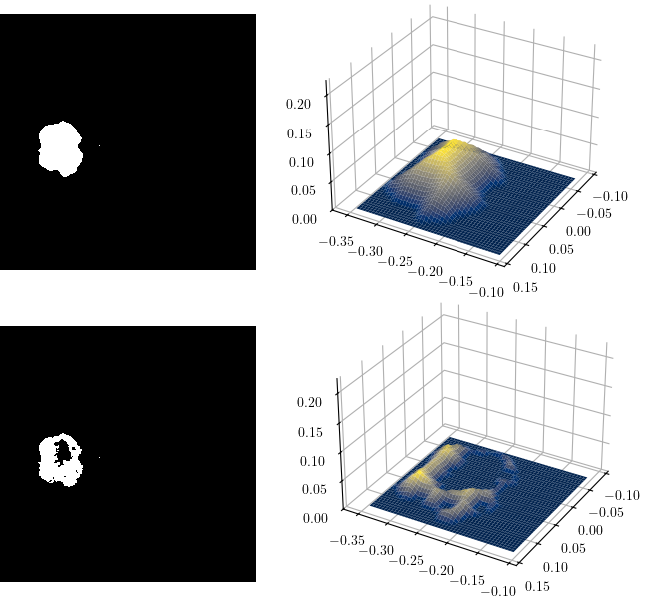}
        \caption{The image in the upper left corner represents the core image of an MRI scan of a glioblastoma, while the bottom left image shows its enhanced version. On the right, a visualization depicts the graphs of the corresponding interior functions.}
        \label{fig:fig1}
    \end{figure}

    Several characteristics of the interior function are displayed in the image. First, note that the function is positive only in regions where the tumour's connected components are located, while it takes a null value elsewhere. Furthermore, the interior function of the enhanced image exhibits lower values than those of the interior function of the core image, and in the case of the enhanced image s, the respective function takes a null value in the space where necrosis is located. This difference serves as the basis for measuring how the image fills its surrounding space.
\end{example}

%%
%%%%%%%%%%%%%%%%%%%%%%%%%%%%%%
%%%%%%%%%%%%%%%%%%%%%%%%%%%%%%
%%%%%%%%%%%%%%%%%%%%%%%%%%%%%%
%%%%%%%%%%%%%%%%%%%%%%%%%%%%%%

\section{Indices for measuring the shape of necrosis}
\label{indices}

By employing the integral of the interior function with respect to subcomplexes we define the following:

\begin{definition}
    Let $K\subseteq \mathbb{R}^n$ be a simplicial complex and $S\subseteq K$ a subcomplex. Consider the associated interior functions $I_K$ and $I_S$. Let us define subcomplex lacunarity as
    \begin{equation*}
        \omega(K,S)=\frac{\int_{\mathbb{R}^n}I_S(x)dx}{\int_{\mathbb{R}^n}I_K(x)dx}
    \end{equation*}
    if $\int_{\mathbb{R}^n}I_K(x)dx>0$, otherwise $\omega(K,P)=0$.
\end{definition}

The concept of subcomplex lacunarity forms the foundation of our framework for analysing the geometry of necrotic conglomeration in Glioblastomas. We apply this concept to core, enhanced and necrosis images to obtain indices tailored to understanding the geometry and topology of necrotic regions.

\begin{definition}
    Let $K\subseteq \mathbb{R}^n$ be a simplicial complex, $E,N\subseteq K$ subcomplexes such that $E\cup N=K$. Here, $K$ corresponds to the complex associated to the core image, $E$ to the enhanced image, and $N$ to the necrosis image.
    \begin{enumerate}
        \item We define the necrotic lacunarity index as
        \begin{equation*}
            \eta=\omega(K,N),
        \end{equation*}
        \item and the tumour mass lacunarity as
        \begin{equation*}
            \tau=\omega(K,E).
        \end{equation*}
    \end{enumerate}
    These two indices are what we will refer to as primary indices. We define the secondary indices from the primary indices as follows.
    \begin{enumerate}
        \item We define lacunarity disorder as
        \begin{equation*}
            \sigma=\tau-\eta,
        \end{equation*}
        \item and the lacunarity polarity as
        \begin{equation*}
            \rho=\frac{\tau+\eta}{2}.
        \end{equation*}
    \end{enumerate}
\end{definition}

\begin{example}
    Now that we have defined indices for measuring geometric properties in Glioblastoma MRIs, we proceed by analysing what each index captures. Figure \ref{fig:fig2} presents several thresholded MRIs of glioblastomas alongside their corresponding index values.

    \begin{figure}[h]
        \centering
        \includegraphics[width=0.9\linewidth]{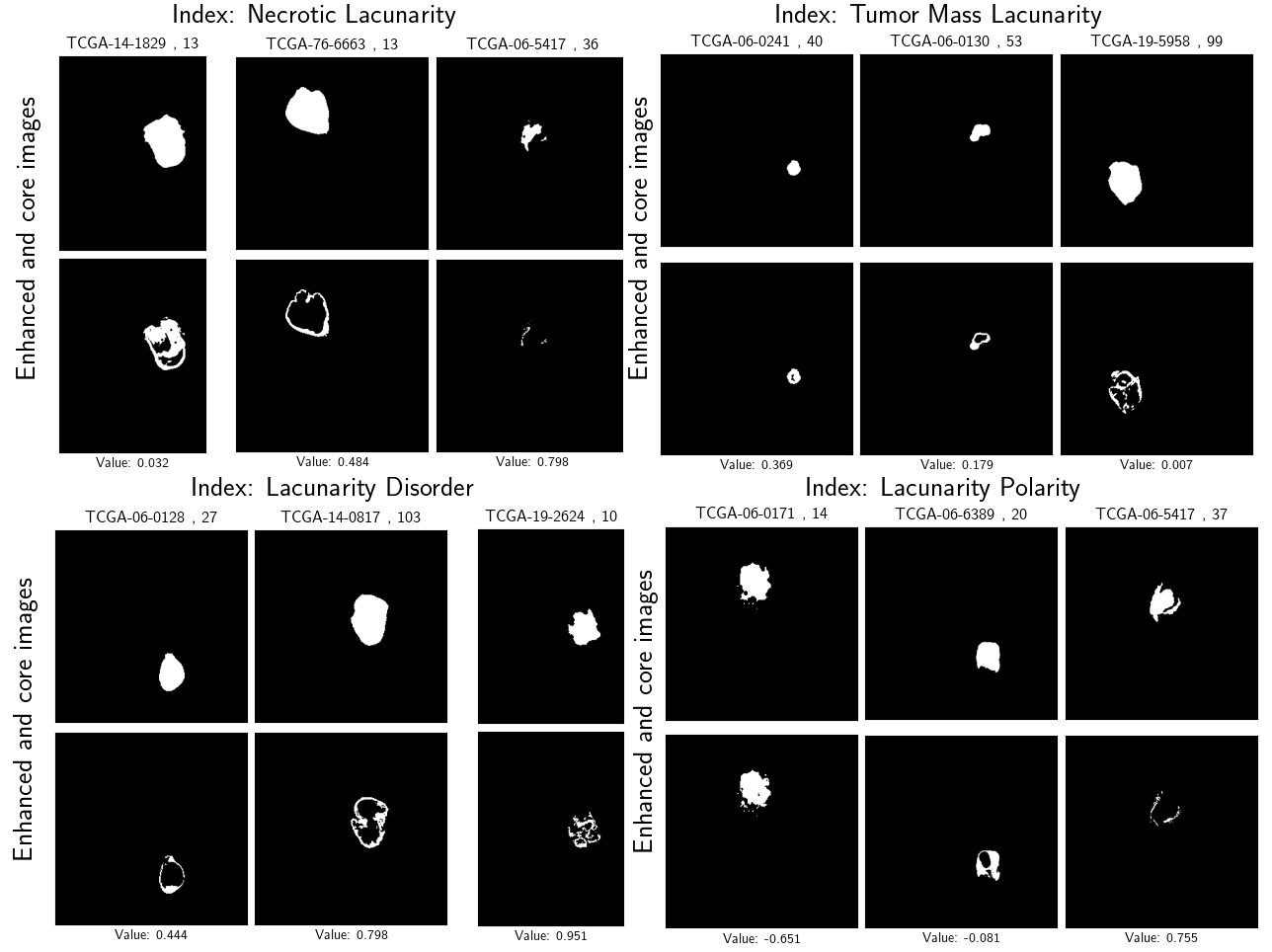}
        \caption{The image presents a table displaying various thresholded MRIs of glioblastomas alongside their corresponding primary and secondary index values. Each table entry contains three pairs of images: the upper row shows core versions, while the lower row displays their corresponding enhanced versions. Above each image pair, the patient tag and image number are indicated. The index values for each pair are displayed at the bottom of the images and are arranged in ascending order from left to right.}
        \label{fig:fig2}
    \end{figure}

    For the images in the necrotic lacunarity section of the table, observe that as the index value increases, the clustered necrotic regions progressively encapsulate a larger portion of the tumour. In the leftmost image pair, multiple necrotic clusters cover approximately 50\% of the total tumour area. The central image depicts a necrotic region consisting of a single cluster that encloses most of the tumour. Finally, the rightmost image shows a tumour where the necrotic region encompasses nearly the entire tumour area.

    The images in the tumour mass lacunarity section of Figure \ref{fig:fig2} illustrate a trend where tumour area clusters become larger as the index value increases. The leftmost image depicts a tumour primarily composed of tumour tissue. In the central image, the necrotic area is more prominent, with a smaller tumour region cluster. The rightmost image shows a tumour where the remaining tumour tissue is concentrated into elongated, thin filaments.

    The lacunarity disorder section of the figure illustrates how, as the index value increases, the tumour becomes increasingly fragmented into smaller clusters of both necrotic and tumour regions. In the leftmost image, two well-defined clusters of necrotic and tumour regions are visible. In the central image, the necrotic region appears more fragmented, a trend that becomes even more pronounced in the rightmost image.

    Finally, the lacunarity polarity section of the figure demonstrates how, as the index value increases, the dominant clusters in the image transition from tumor regions to primarily necrotic regions. Additionally, it should be noted that all index values fall within the range $[0,1]$, except lacunarity polarity which falls within $[-1,1]$. This appears to be an inherent property of these indices.
\end{example}

By constructing a vector consisting of the primary indices corresponding to a set of images, we obtain a point in the space $[0,1]^2$. The location of this point in the space correlates to the geometry of necrosis in the tumour. We will refer to the resulting plot as the {\bf primary index diagram}.

\begin{example}
    Let us analyze how the images from the previous example are positioned within the primary index diagram and how different sections of the diagram correspond to the geometric characteristics of necrosis in the images. Figure \ref{fig:fig3} presents the primary indices of these images, with annotations indicating their respective image tags.

    \begin{figure}[h]
        \centering
        \includegraphics[width=0.6\linewidth]{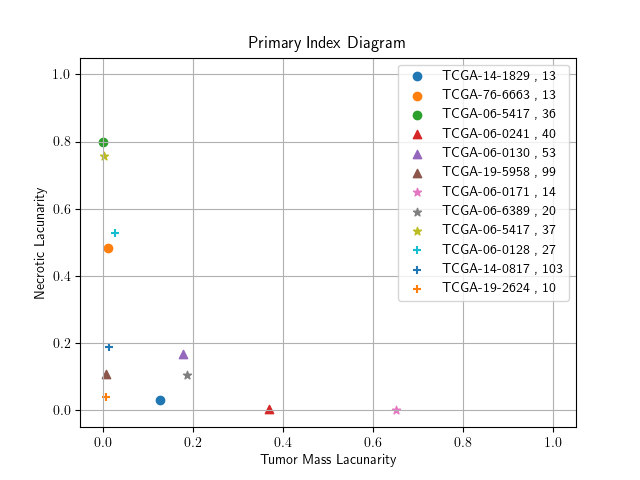}
        \caption{Primary index diagram of thresholded glioblastoma MRIs presented in Example 3.1. The different index groups are distinguished by the shape of the markers in the plot, while the legend indicates the corresponding image tags.}
        \label{fig:fig3}
    \end{figure}

    Note that images with a low degree of tumor mass lacunarity are positioned on the left side of the diagram, while those with a high degree of necrosis are located toward the upper side. A high degree of lacunarity polarity corresponds to points in the upper left corner, whereas a high degree of lacunarity disorder is associated with points near the lower left corner of the diagram.
    
    Images consisting of a single dominant necrotic cluster that encompasses most of the tumor are located in the upper left section of the diagram, including examples such as TCGA-06-5417 (image 36) and TCGA-06-5417 (image 37). Additionally, images with a high degree of necrosis but featuring a more prominent cluster of tumor area are also positioned in this region, though they exhibit lower values of necrotic lacunarity, this includes: TCGA-06-0128 (image 27) and TCGA-70-6663 (image 13).

    Images with an intermediate value of lacunarity polarity are located in the lower left section of the diagram, which also corresponds to a higher degree of lacunarity disorder. This region includes tumors with complex necrotic structures. Images such as TCGA-19-59-5958 (image 99) and TCGA-19-2624 (image 10), which exhibit a high degree of fragmentation, are positioned closest to the lower left corner, indicating a high level of lacunarity disorder. Additionally, this section includes images such as TCGA-06-0130 (image 53) and TCGA-06-6389 (image 20), which display intermediate lacunarity polarity and lower lacunarity disorder, representing tumors with well-defined clusters of both necrotic and tumor regions.

    Finally, the lower right region of the diagram contains images where a large tumor region dominates, encompassing nearly the entire tumor, with only small necrotic clusters present. Examples include TCGA-06-0171 (image 14) and, to a lesser extent, TCGA-06-0241 (image 40).
\end{example}

One of the key strengths of the primary index diagram approach is its versatility, as it can be applied whenever three images are available. In our case, we have two main images, the core and the enhanced image, while the third, the necrosis image, is derived from the first two. However, this technique can be utilized in any scenario where three images are present, as long as they follow the subcomplex structure outlined in Definition 3.2.

%%
%%%%%%%%%%%%%%%%%%%%%%%%%%%%%%
%%%%%%%%%%%%%%%%%%%%%%%%%%%%%%
%%%%%%%%%%%%%%%%%%%%%%%%%%%%%%
%%%%%%%%%%%%%%%%%%%%%%%%%%%%%%

\section{Primary index based clustering analysis}
\label{clustering}

We perform a clustering analysis of the primary index diagram using glioblastoma MRI data from 93 patients. The images analyzed are a subset of the TCGA-GBM dataset, the pictures we will be utilizing have already undergone preprocessing steps of tumor segmentation and image thresholding, more on this in the remarks section. To compute the interior function values, we use the Python software Cubical Ripser \cite{KajiRipser} for persistent homology of cubical complexes and Persim \cite{scikitTDA} for persistence landscape computation. The integral of the interior function, required for calculating the primary indices, is computed using the 2D trapezoidal rule. Due to high computation costs of the algorithms, HPC clusters where used. 

First, we perform cluster analysis on all the images in the dataset, followed by cluster analysis of one of the clusters found in the initial analysis, and we finish by performing cluster analysis using the median vectors per patient as data points. The main tool used for cluster analysis is Silhouette analysis \cite{silhouette}, which we implemented using the Python Library Scikit-Learn \cite{scikitlearn}. To avoid issues caused by a low number of activated pixels in the images (see Remarks section), we excluded the bottom 10\% of images with the lowest number of activated pixels. The resulting dataset consists of 1065 images. Figure \ref{fig:fig4} shows the primary index diagram of the entire dataset. Since the figure reveals no significant outliers, we proceed with the clustering analysis using this dataset.

\begin{figure}[h]
        \centering
        \includegraphics[width=0.6\linewidth]{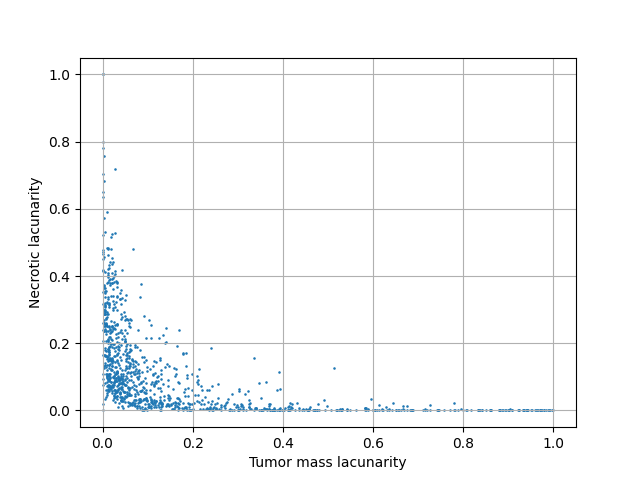}
        \caption{Primary index diagram pertaining to all the data in the dataset. The diagram shows that the data points are contained in $[0,1]^2$, and the data points are heavily situated around the lower and rightmost areas of the diagram.}
        \label{fig:fig4}
    \end{figure}
    
Initially, we computed Silhouette scores for the number of clusters ranging from $3$ to $18$. Table \ref{table:table1} shows the computed Silhouette scores.

\begin{table}[h]
\begin{tabular}{|cllllllll|}
\hline
\multicolumn{1}{|c|}{\textbf{Number of clusters}} & \multicolumn{1}{c|}{3}      & \multicolumn{1}{c|}{4}      & \multicolumn{1}{c|}{5}      & \multicolumn{1}{c|}{6}      & \multicolumn{1}{l|}{7}      & \multicolumn{1}{l|}{8}      & \multicolumn{1}{l|}{9}      & 10     \\ \hline
\multicolumn{1}{|c|}{\textbf{Silhouette score}}   & \multicolumn{1}{c|}{0.4940} & \multicolumn{1}{c|}{0.5435} & \multicolumn{1}{c|}{0.4850} & \multicolumn{1}{c|}{0.4830} & \multicolumn{1}{l|}{0.4645} & \multicolumn{1}{l|}{0.4693} & \multicolumn{1}{l|}{0.4383} & 0.4423 \\ \hline
\multicolumn{9}{|l|}{}                                                                                                                                                                                                                                                       \\ \hline
\multicolumn{1}{|c|}{\textbf{Number of clusters}} & \multicolumn{1}{l|}{11}     & \multicolumn{1}{l|}{12}     & \multicolumn{1}{l|}{13}     & \multicolumn{1}{l|}{14}     & \multicolumn{1}{l|}{15}     & \multicolumn{1}{l|}{16}     & \multicolumn{1}{l|}{17}     & 18     \\ \hline
\multicolumn{1}{|c|}{\textbf{Silhouette score}}   & \multicolumn{1}{l|}{0.4328} & \multicolumn{1}{l|}{0.4515} & \multicolumn{1}{l|}{0.4380} & \multicolumn{1}{l|}{0.4423} & \multicolumn{1}{l|}{0.4403} & \multicolumn{1}{l|}{0.3982} & \multicolumn{1}{l|}{0.4022} & 0.4160 \\ \hline
\end{tabular}
\caption{Table depicting Silhouette scores for number of clusters $n=3,\cdots,18$.}
\label{table:table1}
\end{table}

The table indicates that the best candidates for the number of clusters are $n=3,4,5,6$, with a clearly highest Silhouette score at $n=4$. Figure \ref{fig:silhouette1} shows the Silhouette plots for the selected number of clusters.

\begin{figure}[h]
    \centering
    \includegraphics[width=1\linewidth]{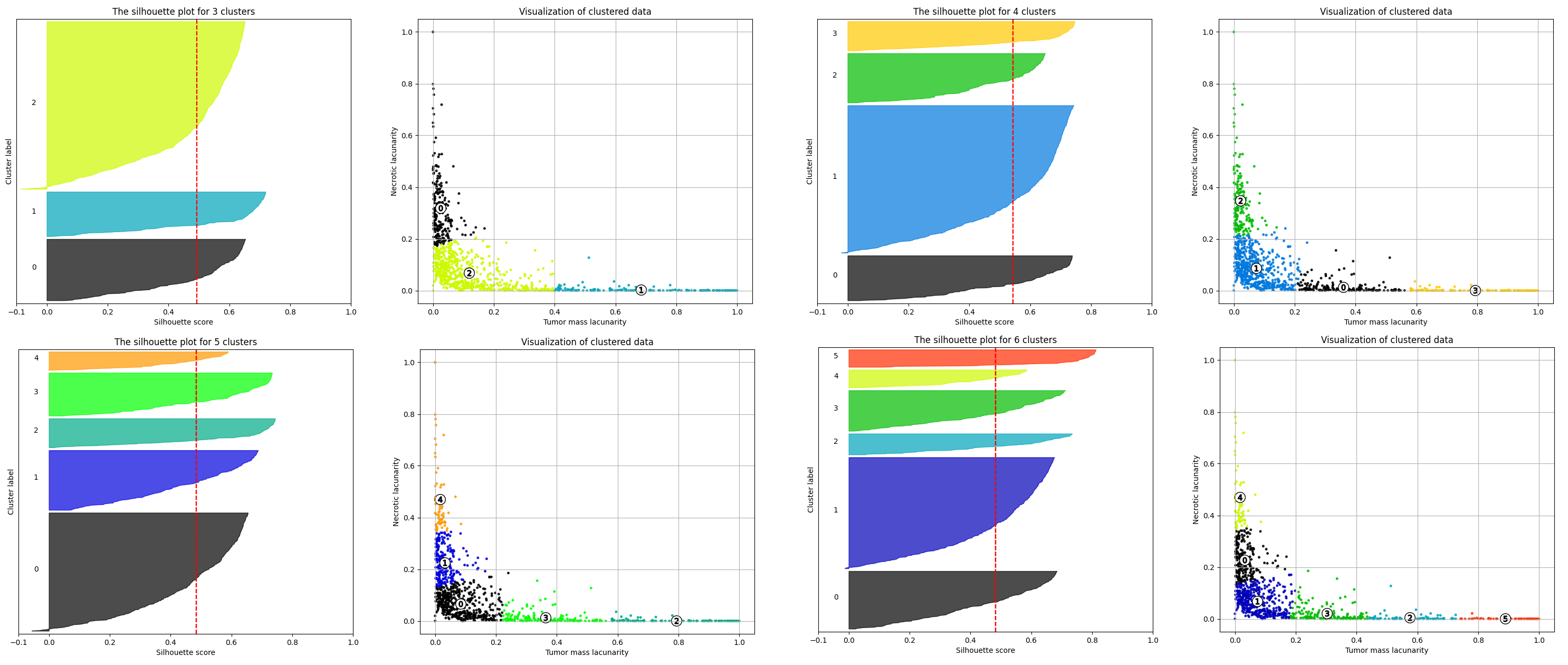}
    \caption{Image depicting Silhouette plots and their respecting cluster plots for number of clusters $n=3,4,5,6$ using K-means.}
    \label{fig:silhouette1}
\end{figure}

The second highest value of Silhouette score is observed at $n=3$. Note that for $n=3$, the Silhouette plot shows a few data points in the primary index diagram with very low Silhouette scores, indicating that the K-means algorithm was unable to properly group those points into a cluster. A similar pattern, though less pronounced, is observed for $n=5$. For $n=4,6$, a few points exhibit negative Silhouette scores; however, $n=4$ stands out with the highest Silhouette score by a significant margin, a value that indicates good performance of K-means algorithm when clustering the data. This analysis suggests that the optimal number of clusters is $n=4$.

Now, let us examine the different clusters and their properties. Figure \ref{fig:cluster_total} depicts the identified clusters.

\begin{figure}[h]
    \centering
    \includegraphics[width=0.4\linewidth]{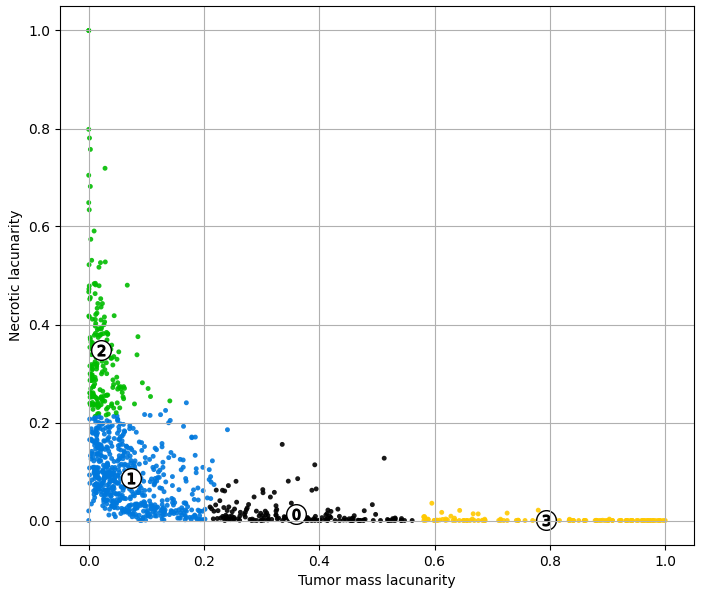}
    \caption{Scatter plot depicting the clusters found for the data in the cohort using K-means, number of clusters $n=4$.}
    \label{fig:cluster_total}
\end{figure}

Cluster 3 is positioned in the lower right corner of the diagram. The images within this cluster correspond to cases with minimal lacunarity polarity, where necrotic regions are primarily located along the tumor's contour. Figure \ref{fig:cluster3} presents representative images from this cluster, alongside visualizations of their respective interior functions.

\begin{figure}[h]
    \centering
    \includegraphics[width=0.5\linewidth]{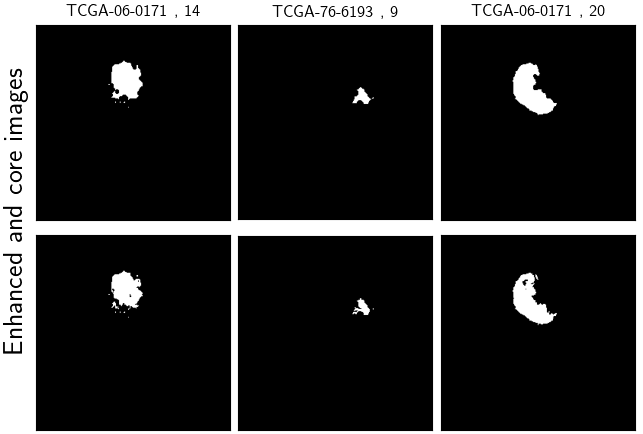}
    \caption{The figure depicts some example of images whose primary indices can be found in cluster 3. The images in the top row are core images while the pictures on the bottom row are the respective enhanced images.}
    \label{fig:cluster3}
\end{figure}

Cluster 0 consists of images with low lacunarity polarity. These images tend to exhibit limited variation in the degree of disorder within the necrotic region. Images in this cluster differentiate themselves from images in cluster 2 by exhibiting larger accumulation of necrosis found in the middle of the tumor, as opposed to the contour. Figure \ref{fig:cluster0} presents representative images from this cluster, along with visualizations of their corresponding interior functions.

\begin{figure}[h]
    \centering
    \includegraphics[width=0.5\linewidth]{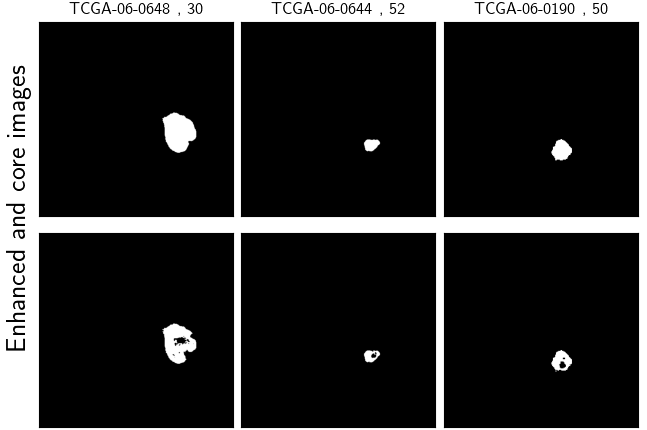}
    \caption{The figure depicts some example of images whose primary indices can be found in cluster 0. The images in the top row are core images while the pictures on the bottom row are the respective enhanced images.}
    \label{fig:cluster0}
\end{figure}

Cluster 1 consists of images with an intermediate level of lacunarity polarity, exhibiting a wide range of disorder within the necrotic region. Figure \ref{fig:cluster1} presents some representative examples from this cluster.

\begin{figure}[h]
    \centering
    \includegraphics[width=0.5\linewidth]{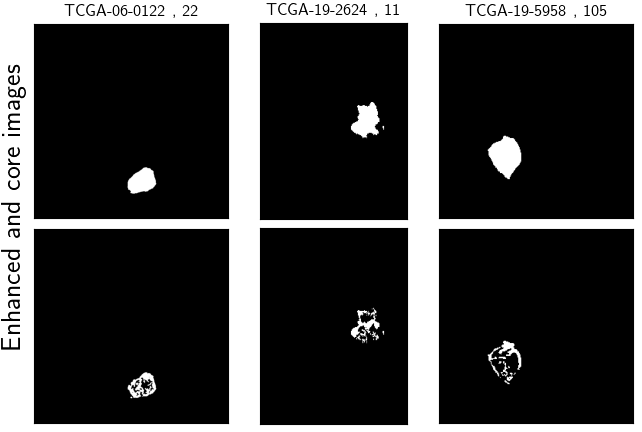}
    \caption{The figure depicts some example of images whose primary indices can be found in cluster 1. The images in the top row are core images while the pictures on the bottom row are the respective enhanced images.}
    \label{fig:cluster1}
\end{figure}

Finally, cluster 2 consists of images with a high degree of lacunarity polarity. As is typical of tumor necrosis, these images show that the necrotic regions tend to cluster at the center of the tumor. Figure \ref{fig:cluster2} presents some representative examples from this cluster.

\begin{figure}[h]
    \centering
    \includegraphics[width=0.5\linewidth]{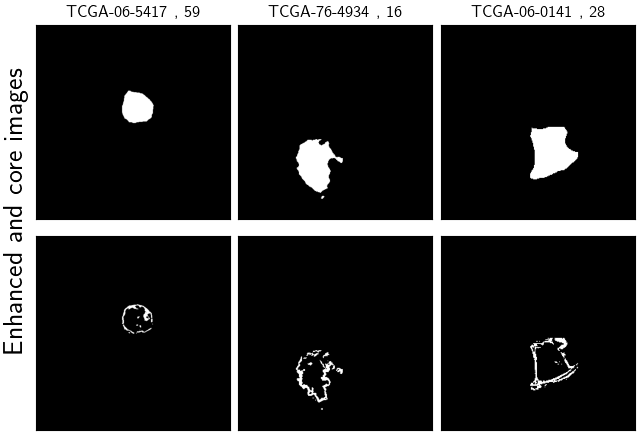}
    \caption{The figure depicts some example of images whose primary indices can be found in cluster 2. The images in the top row are core images while the pictures on the bottom row are the respective enhanced images.}
    \label{fig:cluster2}
\end{figure}

As previously mentioned, the clusters identified by the K-means algorithm are primarily based on lacunarity polarity. Figure \ref{fig:distplot_polarity} presents distribution plots of the lacunarity polarity index for the different clusters. The order of lacunarity polarity follows this pattern: cluster 3 exhibits very low values, cluster 0 shows negative values close to zero, cluster 1 includes both positive and negative values around zero, and cluster 2 has the highest values. This analysis supports the idea that clustering can be effectively based on the degree of lacunarity polarity. Since lacunarity polarity and necrotic fraction are closely related, this also suggests that the main characteristic of the clusters found is the degree of necrosis.

\begin{figure}[h]
    \centering
    \includegraphics[width=0.7\linewidth]{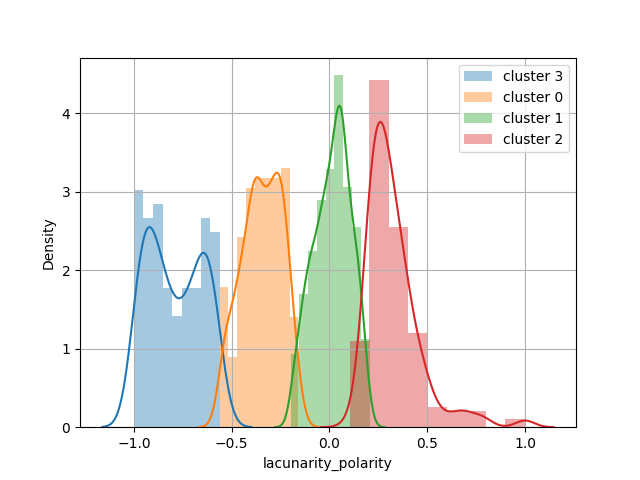}
    \caption{Distribution plot for lacunarity polarity grouped by clusters found in the analysis.}
    \label{fig:distplot_polarity}
\end{figure}

Although the cluster analysis suggests that lacunarity disorder has a limited effect on separating tumor images, we performed a secondary clustering analysis on cluster 1 (images with intermediate lacunarity polarity) to show that disorder and polarity are characteristics that can define necrosis geometry in tumors. Table \ref{table:table2} shows the Silhouette scores for cluster numbers $n = 3, \ldots, 10$.

\begin{table}[h]
\begin{tabular}{|c|c|c|c|c|l|l|l|l|}
\hline
\textbf{Number of clusters} & 3      & 4      & 5      & 6      & 7      & 8      & 9      & 10     \\ \hline
\textbf{Silhouette score}   & 0.4195 & 0.4319 & 0.4092 & 0.3757 & 0.3888 & 0.3940 & 0.3747 & 0.3965 \\ \hline
\end{tabular}
\caption{Table showing the Silhouette scores for clustering analysis of centermost cluster.}
\label{table:table2}
\end{table}

Observe that the computed Silhouette scores are considerably lower than those obtained in the previous analysis, with no score standing out as substantially higher than the others. This indicates the absence of well-defined clusters. Figure \ref{fig:silhouette_new_total} presents the Silhouette plot and the resulting clusters for the three cluster counts with the highest Silhouette scores: $n = 3, 4, 5$.

\begin{figure}[h]
    \centering
    \includegraphics[width=0.9\linewidth]{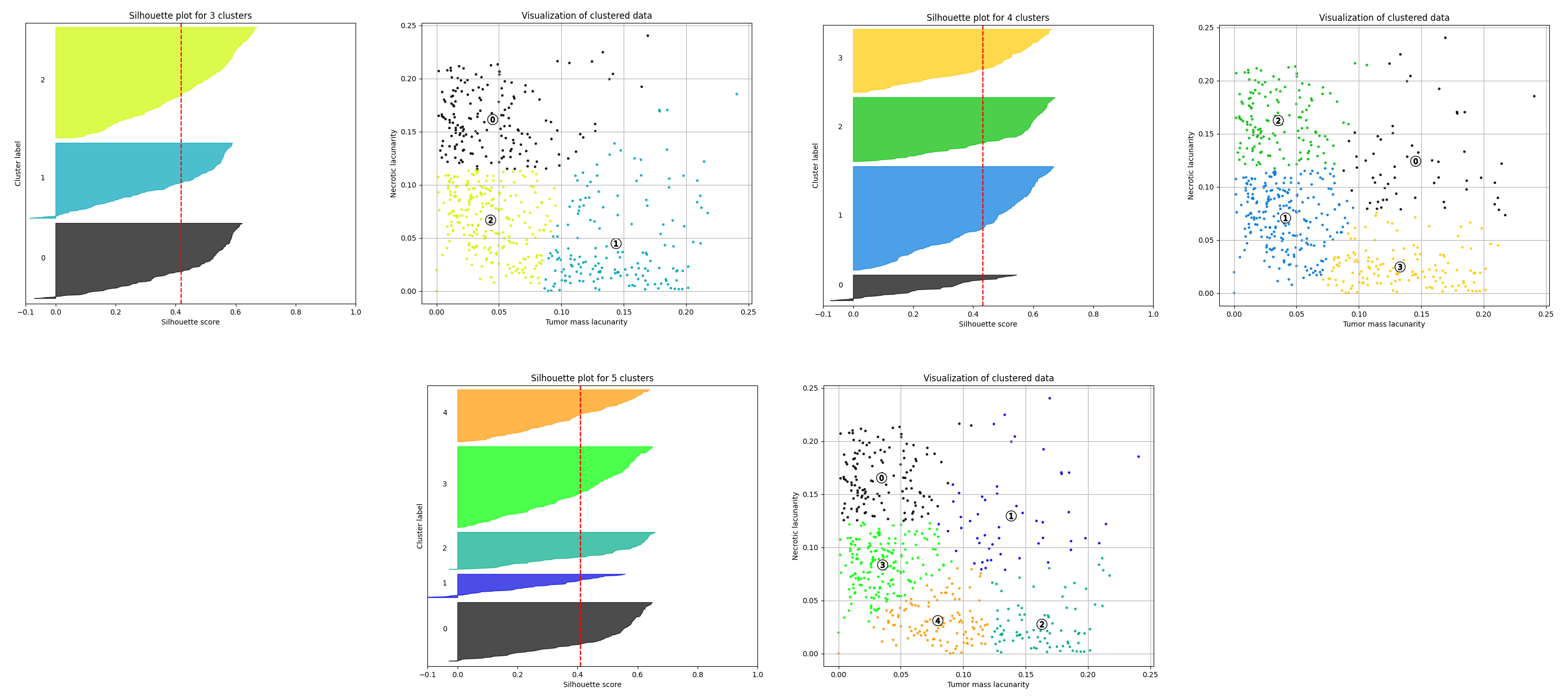}
    \caption{The figure depicts the Silhouette plots alongside the respective scatter plots of the different clusters.}
    \label{fig:silhouette_new_total}
\end{figure}

Regardless of the number of clusters, the Silhouette plots consistently show a few data points with very low Silhouette scores. This suggests that some points cannot be properly clustered using the K-means algorithm. Figure \ref{fig:silhouette_new_total} indicates that the misclassified points tend to be located near the center of the plot. However, selecting four clusters ($n = 4$) yields the highest average Silhouette score and appears to result in the highest number of well-positioned data points. Additionally, for $n = 4$, most clusters (except cluster 0) exhibit relatively high average Silhouette scores, suggesting that the clusters are reasonably well-defined. This analysis supports choosing $n = 4$ as the optimal number of clusters, although we emphasize that this does not conclusively establish that the four clusters are well-defined.

The identified clusters can be described as follows. Cluster 0 includes images with low disorder and an intermediate level of lacunarity polarity. Cluster 1 consists of images with very high disorder and intermediate lacunarity polarity. Cluster 2 contains images with high lacunarity polarity and intermediate disorder. Finally, cluster 3 comprises images with low lacunarity polarity and intermediate disorder. 

\begin{figure}
    \centering
    \includegraphics[width=0.8\linewidth]{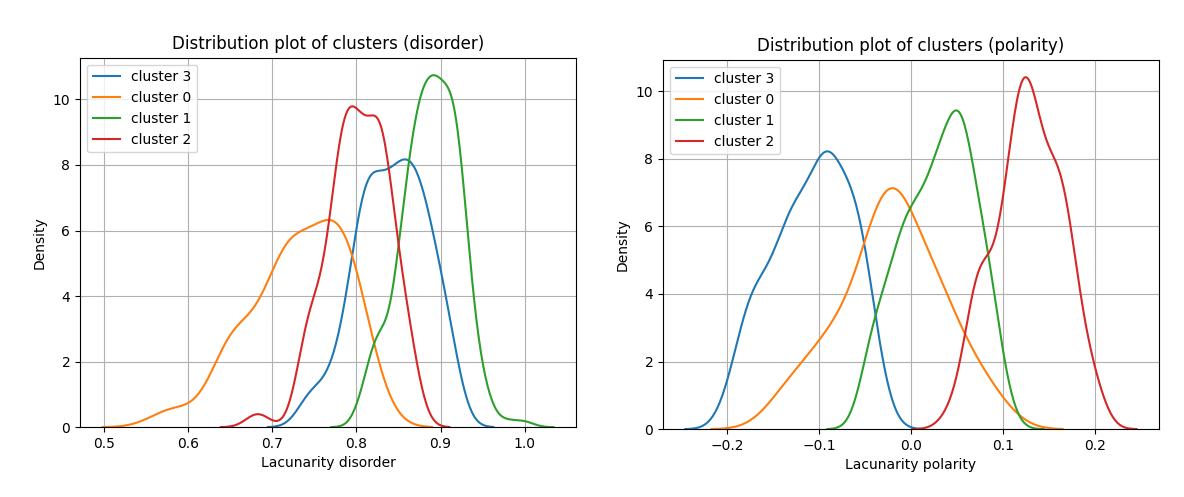}
    \caption{Distribution plots of lacunarity polarity and lacunarity disorder for clusters found in innermost part of primary index diagram.}
    \label{fig:distplot_new}
\end{figure}

Figure \ref{fig:distplot_new} presents distribution plots of lacunarity disorder and the lacunarity polarity index for the different clusters, the figure supports the proposed classification.

Figure \ref{fig:images_cluster_new} presents example images for the different clusters. Cluster 0 tends to include images with well-defined clusters of both necrotic and tumor regions. Cluster 1 contains tumors with similar aggregation patterns of necrosis and tumor regions but exhibits significantly more fragmentation and disorder. Cluster 2 primarily consists of images with larger necrotic clusters than tumor regions, showing relatively low disorder within the necrotic areas. Finally, cluster 3 mostly includes images with high aggregation of tumor regions and some necrotic clusters, but with minimal disorder in the necrotic areas.

\begin{figure}[h]
    \centering
    \includegraphics[width=0.8\linewidth]{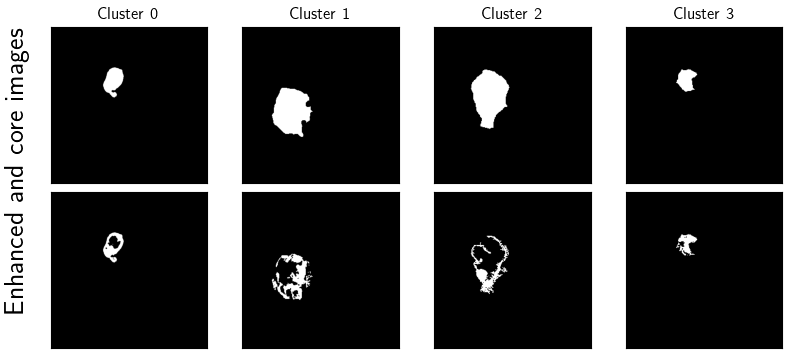}
    \caption{The figure shows some exampled of the clusters found in the analysis for points in the innermost part of the primary index diagram}
    \label{fig:images_cluster_new}
\end{figure}

To conclude this section, we now perform cluster analysis using the median vector per patient on the primary index diagram. The cohort consists of 93 patients. Figure \ref{fig:median_data} displays the data points for each patient. One important characteristic of this figure is that the data points have gathered closer to the lower right corner of the diagram as opposed to before, this indicates that most images per patient tend to be in this area of the diagram.

\begin{figure}[h]
    \centering
    \includegraphics[width=0.6\linewidth]{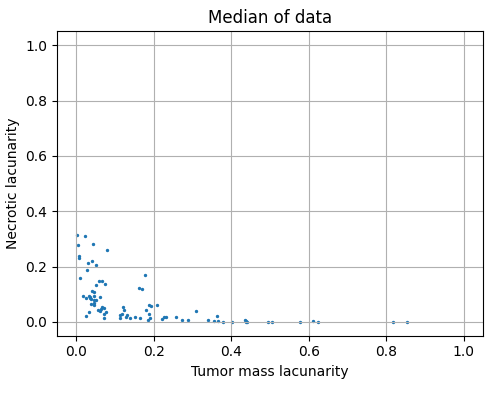}
    \caption{The plot shows the primary index diagram for median vector of primary indices of patient in the cohort.}
    \label{fig:median_data}
\end{figure}

Table \ref{table:table3} shows Silhouette scores for number of cluster equal to $n=3,...10$.

\begin{table}[h]
\begin{tabular}{|c|c|c|c|c|l|l|l|l|}
\hline
\textbf{Number of clusters} & 3      & 4      & 5      & 6      & 7      & 8      & 9      & 10     \\ \hline
\textbf{Silhouette score}   & 0.5233 & 0.5730 & 0.5134 & 0.4960 & 0.4916 & 0.4873 & 0.4758 & 0.5030 \\ \hline
\end{tabular}
\caption{The table shows the Silhouette scored for number of clusters $n=3,\cdots ,10$, for the cluster analysis using median of primary indices.}
\label{table:table3}
\end{table}

Observe that the highest value is reached at $n = 4$, with a value significantly higher than the others. This suggests that the data can be effectively grouped into four clusters. Figure \ref{fig:silhouette_final} shows the Silhouette plot for $n = 4$ clusters. Notably, there appears to be only one data point that is not well grouped into a cluster.

\begin{figure}[h]
    \centering
    \includegraphics[width=0.9\linewidth]{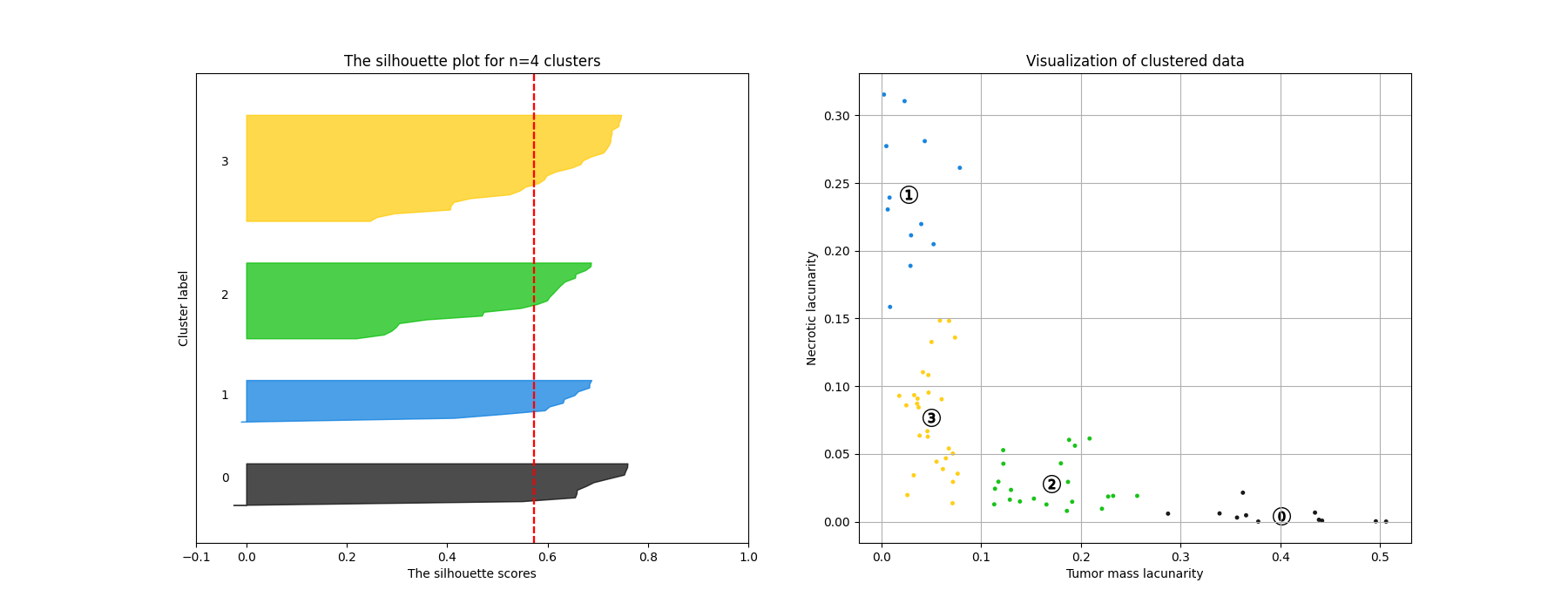}
    \caption{The figure shows the Silhouette plot for $n=4$ cluster alongside the corresponding scatter plot for the found clusters using K-means.}
    \label{fig:silhouette_final}
\end{figure}

This analysis is consistent with the previous clustering results, indicating that the grouping occurs into four clusters, mostly based on the level of lacunarity polarity. This suggests that a classification of lacunarity polarity may be possible for the geometry of necrosis in tumours. However, further research is needed to confirm this conclusion. In particular, to thoroughly validate this finding, it will be important to analyze three-dimensional cubical complexes of tumour images to determine if this pattern persists.

%%
%%%%%%%%%%%%%%%%%%%%%%%%%%%%%%
%%%%%%%%%%%%%%%%%%%%%%%%%%%%%%
%%%%%%%%%%%%%%%%%%%%%%%%%%%%%%
%%%%%%%%%%%%%%%%%%%%%%%%%%%%%%

\section{Remarks}
\label{remarks}

One of the main challenges in this research was the computational cost associated with repeatedly computing persistent homology. The dataset primarily consisted of images with resolutions of $256\times 256$ and $512\times 512$. To compute the integrals of the corresponding interior functions, we employed a $100\times 100$ grid, resulting in computation times of approximately 7 minutes per $256\times 256$ image and 40 minutes per $512\times 512$ image. These results highlight the substantial computational demands of calculating the primary indices, particularly in the absence of high-performance computing resources. 

In the definition of subcomplex lacunarity (def. 3.1), we specified how to handle cases where the denominator equals zero. Ideally, this definition should be applied when the denominator is nonzero, which is expected with sufficiently high image resolution. However, when working with 2D MRI scans of tumors, zero-valued denominators are almost inevitable, since the edges of the tumors give rise to images with very few activated pixels. Therefore, careful monitoring is recommended to prevent significant data distortion.

%%
%%%%%%%%%%%%%%%%%%%%%%%%%%%%%%
%%%%%%%%%%%%%%%%%%%%%%%%%%%%%%
%%%%%%%%%%%%%%%%%%%%%%%%%%%%%%
%%%%%%%%%%%%%%%%%%%%%%%%%%%%%%

\section{Data availability}
\label{sec:data}

As previously mentioned, the data used in this paper is a subset of the TCGA-GBM cohort which can be found at \url{https://portal.gdc.cancer.gov/projects/TCGA-GBM}. The original dataset was obtained from the GitHub repository associated with the paper \cite{SECTGBM} and it available at \url{https://github.com/lorinanthony/SECT.git}. Our dataset, which includes the newly computed necrosis data along with the original data, a Python script with the function to calculate the interior function, and the calculation of the indices can be accessed on our GitHub page \url{https://github.com/fjtellez/TDA-IF-REPO.git}.

%%%%%%%%%%%%%%%%%%%%%%%%%%%%%%
%%%%%%%%%%%%%%%%%%%%%%%%%%%%%%
%%%%%%%%%%%%%%%%%%%%%%%%%%%%%%
%%%%%%%%%%%%%%%%%%%%%%%%%%%%%%

\hfill

\textit{Acknowledgments} \\
Part of this work was conducted while the first author conducted Mitacs GRI project supervised by the second author. We also extend our gratitude to the Digital Research Alliance of Canada for providing access to their HPC clusters. The code used to draw the silhouette plots can be accessed at \href{https://scikit-learn.org/stable/auto_examples/cluster/plot_kmeans_silhouette_analysis.html}{scikit-learn.org}.

%\textit{Disclosure statement:} \\
%The authors do not report potential conflicts of %interest.


\begin{thebibliography}{CJS}

\bibitem{Adams}
Adams, H., et al. (2017). Persistence images: A stable vector representation of persistent homology. Journal of Machine Learning Research, 18(8), 1-35.


\bibitem{Barker}
Barker, F. G., Davis, R. L., Chang, S. M.,  Prados, M. D. (1996). Necrosis as a prognostic factor in glioblastoma multiforme. Cancer: Interdisciplinary International Journal of the American Cancer Society, 77(6), 1161-1166.

\bibitem{BubenikSTDA}
Bubenik, P. Statistical topological data analysis using persistence landscapes. J. Mach. Learn. Res. 16.1 (2015): 77-102.

\bibitem{BubenikCAT}
Bubenik, P., Scott, J.A. Categorification of Persistent Homology. Discrete Comput Geom 51, 600–627 (2014).

\bibitem{Carlsson}
Carlsson, G.,  Vejdemo-Johansson, M. (2021). Topological data analysis with applications. Cambridge University Press.


\bibitem{Curtin}
Curtin, L., Whitmire, P., White, H. et al.
Shape matters: morphological metrics of glioblastoma imaging abnormalities as biomarkers of prognosis. 
Sci Rep 11, 23202 (2021). 
https://doi.org/10.1038/s41598-021-02495-6

\bibitem{SECTGBM}
Crawford, L., Monod, A., Chen, A. X., Mukherjee, S., Rabadán, R. (2019). Predicting Clinical Outcomes in Glioblastoma: An Application of Topological and Functional Data Analysis. Journal of the American Statistical Association, 115(531), 1139–1150.


\bibitem{Hatcher}
Hatcher, Allen. Algebraic topology. Cambridge University Press, 2002.

\bibitem{KajiRipser}
Kaji, S., Sudo, T.,  Ahara, K. (2020). Cubical ripser: Software for computing persistent homology of image and volume data. arXiv preprint arXiv:2005.12692.


\bibitem{fraction}
Khalil, A. A., Horsman, M. R.,  Overgaard, J. (1995). The importance of determining necrotic fraction when studying the effect of tumour volume on tissue oxygenation. Acta Oncologica, 34(3), 297-300.

\bibitem{Meng}
Meng, K., Wang, J., Crawford, L., Eloyan, A. Randomness of Shapes and Statistical Inference on Shapes via the Smooth Euler Characteristic Transform. Journal of the American Statistical Association, 1–25 (2024). https://doi.org/10.1080/01621459.2024.2353947

\bibitem{scikitlearn}
Pedregosa et al. Scikit-learn: Machine Learning in Python. JMLR 12, pp. 2825-2830, 2011.

\bibitem{Blumberg}
Rabadan R., Blumberg A.J. Topological Data Analysis for Genomics and Evolution: Topology in Biology. Cambridge University Press; 2019.

\bibitem{Raza}
Raza, S., et al. Necrosis and Glioblastoma: A Friend or a Foe? A Review and a Hypothesis. Neurosurgery 51(1):p 2-13, July 2002. 

\bibitem{silhouette}
Rousseeuw, P.
Silhouettes: A graphical aid to the interpretation and validation of cluster analysis,
Journal of Computational and Applied Mathematics,
Volume 20,
1987, 53--65,
https://doi.org/10.1016/0377-0427(87)90125-7

\bibitem{scikitTDA}
Saul, N., Tralie, C. (2019). Scikit-TDA: Topological Data Analysis for Python. Zenodo. http://doi.org/10.5281/zenodo.2533369

\bibitem{StabWass}
Skraba, P., Turner, K. (2020). Wasserstein stability for persistence diagrams. arXiv preprint arXiv:2006.16824.

\bibitem{PHT}
Turner, K., Mukherjee, S.,  Boyer, D. M. (2014). Persistent homology transform for modeling shapes and surfaces. Information and Inference: A Journal of the IMA, 3(4), 310-344.

\bibitem{Yee}
Yee, P.P., Wei, Y., Kim, SY. et al. Neutrophil-induced ferroptosis promotes tumor necrosis in glioblastoma progression. Nat Commun 11, 5424 (2020). https://doi.org/10.1038/s41467-020-19193-y



\end{thebibliography}
\end{document}